\documentclass{amsart}

\usepackage[english]{babel}

\newtheorem{theorem}{Theorem}[section]
\newtheorem{corollary}[theorem]{Corollary}

\newtheorem{lemma}[theorem]{Lemma}
\newtheorem{proposition}[theorem]{Proposition}
\theoremstyle{definition}
\newtheorem{ex}[theorem]{Example}
\newenvironment{example}
	{
	\pushQED{\qed}\begin{ex}}
	{\popQED\end{ex}}
	
\numberwithin{equation}{section}

\def\N{\mathbb{N}}
\def\R{\mathbb{R}}
\def\C{\mathbb{C}}
\def\Z{\mathbb{Z}}

\def\Rng{\mathcal{R}}
\def\Ker{\mathcal{N}}
\def\H{\mathcal{H}}
\newcommand{\B}[1][\H]{\boldsymbol{B}(#1)}
\newcommand{\bspace}[1][\H]{\ell_2(\Z, #1)}

\def\dim{\operatorname{dim}}

\def\spec{\sigma}

\newcommand{\bshift}[1]{#1 \sim \{#1_n\}_{n\in\Z}}
\newcommand{\ushift}[1]{#1 \sim \{#1_n\}_{n\in\N}}

\newcommand{\ue}{\cong}

\newcommand{\hlf}{\frac{1}{2}}

\begin{document}

\title[Unitary equivalence of bilateral shifts]{On unitary equivalence of bilateral operator valued weighted shifts}
\author{Jakub Ko\'smider}
\address{Instytut Matematyki, Uniwersytet Jagiello\'nski, ul. \L{}ojasiewicza 6, PL-30348 Kra\-k\'ow, Poland}
\email{jakub.kosmider@im.uj.edu.pl}

\thanks{}
\dedicatory{}



\begin{abstract}
We establish a characterization of unitary equivalence of two bilateral operator valued weighted shifts with quasi-invertible weights by an operator of diagonal form.
We also present an example of unitary equivalence between shifts defined on $\mathbb{C}^2$ which cannot be given by any unitary operator of diagonal form.
The paper is concluded with investigation of unitary operators than can give unitary equivalence of bilateral operator valued weighted shifts.
\end{abstract}
\maketitle

\section{Introduction and preliminaries}
Classical weighted shift operators and their properties have already been studied for a long time by many authors (see, e.g., \cite{shields, athavale, pietrzycki, geher}).
By classical weighted shifts we understand both unilateral and bilateral weighted shifts defined on $\C$. 
There are many papers devoted to problems of weighted shifts in more general context in which these operators are defined on arbitrary Hilbert spaces (see \cite{lambert, guyker, ivanovski, li_ji_sun, bourhim, jablonski}).
In some of them authors give or use results concerning unitary equivalence (see \cite{lambert, guyker, orovcanec, li_ji_sun, jablonski}).
Jab\l o\'nski, Jung and Stochel introduced in \cite{stochel} the class of weighted shifts on directed trees, which generalizes unilateral and bilateral shifts with classical weights.

Unitary equivalence of unilateral operator valued weighted shifts with invertible weights defined on arbitrary Hilbert space was characterized by Lambert in \cite[Corollary 3.3]{lambert}.
Orov\v{c}anec provided in \cite[Theorem~1]{orovcanec} characterization in case shifts have quasi-invertible weights. 
This result was later proved with weaker assumptions, namely, for unilateral shifts with weights having dense ranges by Anand, Chavan, Jab\l o\'nski, Stochel in \cite[Theorem~2.3]{cauchy}.
Jab\l o\'nski proved in \cite[Propositon~2.2]{jablonski} that unilateral operator valued weighted shift with invertible weights is unitarily equivalent to unilateral operator valued weighted shift with weights $\{T_n\}_{n=0}^\infty$ such that product $T_n\dots T_0$ is a positive operator for all $n\in\N$. 

The are some partial results regarding unitary equivalence of bilateral operator valued weighted shifts.
Li, Ji and Sun proved that each bilateral operator valued weighted shift with invertible weights defined on $\C^m$ for $m\ge 2$ is unitarily equivalent to a shift with upper triangular weights (see \cite[Theorem~2.1]{li_ji_sun}).
Shields provided in \cite{shields} characterization of unitary equivalence in case of classical bilateral shifts. 
Guyker proved in \cite{guyker} a result regarding unitary equivalence of bilateral operator valued weighted shift with the one having positive weights.
The proof required additional assumption i.e., normality and commutativity of weights.

In what follows, we denote by $\N$, $\N_+$, $\Z$, $\R$, $\R_+$ and $\C$ the sets of non-negative integers, positive integers, integers, real numbers, non-negative real numbers and complex numbers, respectively. 
Throughout the paper by $\H$ we denote a nonzero complex Hilbert space. 
The symbol $\B$ stands for the $\mathcal{C}^*$-algebra of all bounded operators defined on $\H$.
All operators considered in this paper are assumed to be linear.
By $\Rng(A)$, $\Ker(A)$ and $\spec(A)$ we understand the range, the kernel and the spectrum of operator $A\in\B$, respectively.
As usual, $I\in\B$ stands for the identity operator. 
Unitary equivalence of operators $A$ and $B\in\B$ will be denoted by $A \ue B$. 
We also write $A \ue_U B$ to emphasize that unitary equivalence is given by $U$. 
For a closed subspace $M$ of $\H$, by $M^{\perp}$ we denote its orthogonal complement.
If $M$ and $N$ are two closed subspaces of $\H$, which are orthogonal, then we write $M\perp N$.
We say that an operator $A\in\B$ is \textit{quasi-invertible}, if $A$ is injective and $\overline{\Rng(A)} = \H$.
The reader can verify that, if $A\in\B$ is quasi-invertible, then so is $A^*$.
For a positive operator $A\in\B$ we denote by $A^\hlf$ the (positive) square root of $A$. 
Operator $A\in\B$ is called a \textit{partial isometry} if $||Ax|| = ||x||$ for all $x\in \Ker(A)^{\perp}$.
The following result is well known and it can be found in \cite[Exercise VIII.3.15]{conway}. 

\begin{lemma}
\label{lemma_isometries}
Let $A\in\B$. Then the following are equivalent:
\begin{enumerate}
	\item $A$ is a partial isometry,
	\item $A^*$ is a partial isometry,
	\item $A^*A$ is the orthogonal projection onto $\Ker(A)^{\perp}$,
	\item $AA^*$ is the orthogonal projection onto $\Rng(A)$,
	\item $AA^*A = A$,
	\item $A^*AA^* = A^*$.
\end{enumerate}
\end{lemma}

Despite of the fact that the following lemma is definitely folklore, we will state it for the reader's convenience, as we will refer to it later. 

\begin{lemma}
\label{lemat_unitarne}
Assume that $S$, $T \in \B$ have dense ranges.
If $||Sx|| = ||Tx||$ for all $x \in \H$, then there exists unitary operator $V$ 
on $\H$ such that $VS = T$.
\end{lemma}

We define a Hilbert space $\bspace$ as the space
$\oplus_{n\in \Z} \H$ equipped with the inner product defined by
$\langle x,y\rangle = \sum_{i = -\infty}^{\infty} \langle x_i, y_i\rangle_{\H}$ for  $x,y \in \bspace$.
This space consists of all vectors $x = (\dots, x_{-1}, \boxed{x_0}, x_1, \dots)$ satisfying  
$\sum_{n=-\infty}^{\infty} ||x_n||^2_{\H} < \infty$, where $\boxed{\cdot}$ denotes the $0$th element of $x$. 
Operator $U\in\B[\bspace]$ can be expressed as infinite matrix $[U_{i,j}]_{i,j\in\Z}$, where $U_{i,j} \in \B$ for all $i,j\in\Z$.

We say that $S \in \B[\bspace]$ is a \textit{diagonal operator} if
there exists a two-sided sequence of operators $\{S_n\}_{n\in \Z} \subseteq \B$ such that $\{||S_n||\}_{n\in\Z}$ is bounded and  
$$
S(\dots, x_{-1}, \boxed{x_0}, x_1, \dots) = (\dots, S_{-1}x_{-1}, \boxed{S_0x_0}, S_1x_1, \dots), \quad x \in \bspace.
$$

Let $\{S_n\}_{n\in \Z} \subseteq \B$ be a two-sided sequence of nonzero operators such that $\{||S_n||\}_{n\in\Z}$ is bounded.
We define $S\in\B[\bspace]$ by
$$
	S(..., x_{-1}, \boxed{x_0}, x_1, ...) = 
	(..., S_{-1}x_{-2}, \boxed{S_{0}x_{-1}}, S_{1}x_{0}, ...), \quad x \in \bspace.
$$
Operator $S$ is called a \textit{bilateral operator valued weighted shift} on $\H$ with operator weights $\{S_n\}_{n\in\Z}$ and it will be denoted by $\bshift{S}$.
Denote by $F$ the unitary bilateral operator valued weighted shift with all weights being identity operators on $\H$.
We say that an operator $S \in \B[\bspace]$ is \textit{of diagonal form} if there exist $k\in\Z$ and a diagonal operator $T \in \B[\bspace]$ such that $S = F^k T$.

Let $\bshift{S}$. We can represent $S$ by the following infinite matrix 
$$S =
\begin{bmatrix}
 \ddots & \ddots &  \ddots &  \ddots   & \ddots  \\
 \ddots & S_{0} & \boxed{0} & 0 & \ddots \\
 \ddots & 0  & S_1 & 0 & \ddots  \\
 \ddots & 0 & 0 & S_2 & \ddots \\
 \ddots & \ddots & \ddots& \ddots& \ddots
\end{bmatrix}
$$
where $\boxed{\cdot}$ 
indicates the element indexed by $(0,0)$.
It is worth noting that, as opposed to \cite{lambert}, we do not assume that weights of $S$ are invertible.

In this paper we focus on the problem of unitary equivalence of bilateral operator valued weighted shifts with quasi-invertible weights.
The paper is organized as follows. 
In Section~2 we investigate unitary equivalence given by operators of diagonal form.
Corollary~\ref{corollary_main_theorem} establishes the characterization of unitary equivalence of bilateral operator valued weighted shifts with quasi-invertible weights given by an operator of diagonal form.
In Theorem~\ref{positive_equivalence} we prove that each bilateral operator valued weighted shift with quasi-invertible weights is unitarily equivalent to a bilateral weighted shift having positive weights. 
We conclude this section with proving that bilateral operator valued weighted shift having normal and commuting weights defined on $\C^m$ for $m\ge 2$ is unitarily equivalent to a bilateral weighted shift with weights being diagonal operators (see Proposition~\ref{propo}).
 
Section~3 is devoted to the problem of unitary equivalence given by operators that are not of diagonal form and to investigation of unitary operators on $\bspace$ that can give unitary equivalence of weighted shifts.
We begin it with Example~\ref{unitary_equiv_ex} that shows two bilateral operator valued weighted shifts defined on $\C^2$ which are unitarily equivalent, but the unitary equivalence is not given by any operator of diagonal form.
Proposition~\ref{proposition_two_diagonals} states that, if $U\in\B[\bspace]$ contains exactly two nonzero diagonals and all other elements of $U$ are zero operators, then the operators on these diagonals are partial isometries.
We also investigate unitary operators that give unitary equivalence of bilateral weighted shifts defined on $\C^m$ for $m\ge 2$.
Proposition~\ref{proposition_isometry} states that under some additional assumptions, if $U\in\B[\bspace[\C^2]]$ is unitary and all elements of $U$ except for three diagonals are zero operators, then one of the diagonals contains only zero operators.

Finally, Section~4 contains final remarks and concludes some open problems related to unitary equivalence of bilateral weighted shifts.

\section{Unitary equivalence given by an operator of diagonal form}

In this section we present results related to unitary equivalence of bilateral operator valued weighted shifts given by an operator of diagonal form.
It contains also some general facts which usage is not limited to this section.

We will begin with stating the following key lemma required for further references,
which is a two-sided counterpart of \cite[Lemma]{orovcanec} (see also \cite[Lemma~2.1]{lambert} and \cite[Proposition~5 (a)]{shields}).
Its proof is left to the reader.

\begin{lemma}
\label{lemat_podstawowy}
Let $\bshift{S}$, $\bshift{T}$ and $S_n$, $T_n$ be quasi-invertible for each $n\in\Z$. 
Assume that $A\in\B[\bspace]$. 
Then the following are equivalent:
\begin{enumerate}
\item $AS = TA$,
\item $A_{i+1,j+1} S_j = T_i A_{i,j}$ for each $i,j\in\Z$.
\end{enumerate}
\end{lemma}

There is a significant difference between \cite[Lemma]{orovcanec} and the one presented above. 
In the case of unilateral weighted shifts every vector in the range of a shift has the zero as the first element.
Hence, each operator intertwining two unilateral weighted shifts has a triangular matrix. 
In the case of bilateral weighted shifts equality $AS = TA$ does not imply triangularity of $A$ (see Example \ref{unitary_equiv_ex} below).

Lemma~\ref{lemat_podstawowy} gives the following important result.

\begin{corollary}
\label{corollary_podstawowy}
Let $\bshift{S}$, $\bshift{T}$ and $S_n$, $T_n$ be quasi-invertible for each $n\in\Z$. 
Assume that $A\in\B[\bspace]$ be such that $AS=TA$. 
If $A_{i,j} \ne 0$ for some $i,j\in\Z$, then $A_{i+n,j+n} \ne 0$ for all $n\in\Z$.
\end{corollary}

It follows from Corollary \ref{corollary_podstawowy} that the unitary operator
\newcommand{\sqc}{\frac{1}{\sqrt{2}}}
$$
\begin{bmatrix}
 \ddots & \ddots	& \ddots 	   & \ddots 		& \ddots 	& \ddots& \ddots  \\
 \ddots & I			& 0 	 	   & 0		 		& 0		 	& 0     & \ddots  \\
 \ddots & 0			& \sqc I	   & 0 				& \sqc I  	& 0 	& \ddots  \\
 \ddots & 0			& 0 		   & \boxed{I} 		& 0 		& 0 	& \ddots  \\
 \ddots & 0			& \sqc I	   & 0 				& -\sqc I 	& 0 	& \ddots  \\
 \ddots & 0			& 0 		   & 0 				& 0 		& I		& \ddots  \\
 \ddots & \ddots	& \ddots 	   & \ddots			& \ddots	& \ddots& \ddots
\end{bmatrix}
$$
does not give unitary equivalence between any two bilateral operator valued weight\-ed shifts with quasi-invertible weights.

The following theorem gives a necessary and sufficient condition for two bilateral operator valued weighted shifts with quasi-invertible weights to be unitarily equivalent by an operator of diagonal form.
Proof of this fact is based on the proof of similar result for unilateral operator valued weighted shifts from \cite[Theorem~2.3]{cauchy} (see also \cite[Theorem~1]{orovcanec}).

\begin{theorem}
\label{main_theorem}
Let $\bshift{S}$, $\bshift{T}$ and $m\in\Z$ be such that $S_{m+n}$, $T_n$, $S_{m-n-1}^*$ and $T_{-n-1}^*$ have dense ranges for $n\in\N$. 
Then the following are equivalent 
\begin{enumerate}
\item there exists $U\in\B[\bspace]$ of diagonal form such that $S \ue_U T$ and $U_{0,m}\ne 0$,
\item there exists unitary operator $U_{0,m}\in\B$ such that the following hold:
	\begin{enumerate}
	\item $||S_{m+n-1}\dots S_mx|| = ||T_{n-1}\dots T_0U_{0,m}x||$ for all $x\in \H$ and $n \in \N_+$,
	\item $||S_{m-n}^*\dots S_{m-1}^*x|| = ||T_{-n}^*\dots T_{-1}^*U_{0,m}x||$ for all $x\in \H$ and $n \in \N_+$.
	\end{enumerate}
\end{enumerate}
\end{theorem}
\begin{proof}
(i) $\Rightarrow$ (ii).
Assume that $S \ue_U T$, where $U\in\B[\bspace]$ is of diagonal form.
Let $n \in \N_+$. Then, by Lemma~\ref{lemat_podstawowy},
$$
U_{n,m+n}S_{m+n-1}\dots S_{m} = T_{n-1}\dots T_0U_{0,m}.
$$
which implies (a).
Let us now check that (b) also holds. 
Let $n \in \N_+$. 
Again, by Lemma~\ref{lemat_podstawowy},
$$
U_{0,m}S_{m-1}\dots S_{m-n} = T_{-1}\dots T_{-n}U_{-n,m-n}
$$
which is equivalent to the following
$$
S_{m-1}\dots S_{m-n}U_{-n,m-n}^* = U_{0,m}^*T_{-1}\dots T_{-n}.
$$
After taking adjoints we get that for all $x \in \H$ and $n \in \N_+$ it is true that
$$
||S_{m-n}^*\dots S_{m-1}^*x|| = ||T_{-n}^*\dots T_{-1}^*U_{0,m}x||,
$$
which proves (b).

(ii) $\Rightarrow$ (i).
We will construct $U\in\B[\bspace]$ of diagonal form with unitary operators $U_{n,m+n}\in\B$ on its diagonal, which satisfy the following 
\begin{equation}
\label{1233322}
	U_{n+1,m+n+1}S_{m+n} = T_{n}U_{n,m+n}, \quad n\in\Z .
\end{equation}
In order to simplify formulas we introduce notation $V_{n} := U_{n,m+n}$ for $n\in\Z$.

We will begin with constructing operators $V_n$ for $n\in\N_+$.
Since $S_m$ and $T_0V_0$ have dense ranges and (a) holds with $n=1$, then, by Lemma~\ref{lemat_unitarne}, there exists unitary $V_1$ such that 
$V_1S_{m} = T_{0}V_0$.
Now, assume that $n>1$ and unitary operators $V_1,\dots,V_n$ are already defined to be such that 
$V_{i+1}S_{m+i} = T_{i}V_{i}$ for $i\in\{1,\dots,n-1\}$.
Again, we use Lemma~\ref{lemat_unitarne} for operators $S_{m+n}\dots S_{m}$ and $T_{n}\dots T_0V_0$, which have dense ranges and get that there exists a unitary operator $V_{n+1}$ such that 
$$
V_{n+1}S_{m+n}\dots S_{m} = T_{n}\dots T_{0}V_0.
$$
By the above we see that 
\begin{equation}
\label{sd9h}
(V_{n+1}S_{m+n} - T_nV_{n})S_{m+n-1}\dots S_{m} = V_{n+1}S_{m+n}\dots S_{m} - T_{n}\dots T_{0}U_0 = 0.
\end{equation}
Since $S_{m+n-1}\dots S_{m}$ has dense range, \eqref{sd9h} implies that 
$V_{n+1}S_{m+n} = T_nV_{n}$.

We will now focus on finding operators $V_{-n}$ for $n\in\N_+$.
We begin with definition of $V_{-1}$.
Since $S_{m-1}^*$ and $T_{-1}^*U_{0,m}$ have dense ranges and (b) holds for $n=1$, by Lemma~\ref{lemat_unitarne}, there exists a unitary $V_{-1}$ such that
$V_{-1}S_{m-1}^* = T_{-1}^*U_0$.
This implies that $V_0S_{m-1}=T_{-1}V_{-1}$.
Let $n>1$. Assume that $V_{-1},\dots,V_{-n+1}$ are already defined unitary operators on $\H$ such that $V_{-i+1}S_{m-i} = T_{-i}V_{-i}$ for 
$i \in \{1,\dots,n-1\}$.
We will construct $V_{-n}$ such that $V_{-n+1}S_{m-n} = T_{-n}V_{-n}$.
It is enough to find $V_{-n}$ so that the following holds
$$
V_{-n}S_{m-n}^*\dots S_{m-1}^* = T_{-n}^*\dots T_{-1}^*V_0,
$$
because then we will get the following equality
$$
V_0S_{m-1}\dots S_{m-n} = T_{-1}\dots T_{-n}V_{-n}.
$$
We get $V_{-n}$ by using Lemma~\ref{lemat_unitarne} for operators $S_{m-n}^*\dots S_{m-1}^*$ and $T_{-n}^*\dots T_{-1}^*V_0$ with dense ranges.
Now, we only need to show that $V_{-n+1}S_{m-n} = T_{-n}V_{-n}$. 
We will do this by proving that $V_{-n}S_{m-n}^* = T_{-n}^*V_{-n+1}$, which is an equivalent condition.
Let us consider the following:
$$
(V_{-n}S_{m-n}^* - T_{-n}^*V_{-n+1})S_{m-n+1}^*\dots S_{m-1}^* = 
	V_{-n}S_{m-n}^*\dots S_{m-1}^* - T_{-n}^*\dots T_{-1}^*V_0 = 0.
$$
Since $S_{m-n+1}^*\dots S_{m-1}^*$ has dense range, $V_{-n}S_{m-n}^* = T_{-n}^*V_{-n+1}$. 

We constructed sequence $\{U_{n,m+n}\}_{n\in\Z}$ of unitary operators such that $\eqref{1233322}$ holds.
By Lemma~\ref{lemat_podstawowy} it is true that $S\ue_UT$, where $U$ is of diagonal form.
This completes the proof.
\end{proof}

It is worth noting that, if we additionally assume that $S$ and $T$ have quasi-invertible weights in Theorem~\ref{main_theorem}, then we can choose any other operator $U_{k,m+k}$ instead of $U_{0,m}$ for $k\in\Z$ and modify the statement. 
In this way we get the following result.

\begin{corollary}
\label{corollary_main_theorem}
Let $\bshift{S}$, $\bshift{T}$ have quasi-invertible weights and let $m\in\Z$. 
Then the following are equivalent
\begin{enumerate}
\item there exists $U\in \B[\bspace]$ of diagonal form such that $S \ue_U T$ and $U_{0,m}\ne 0$,
\item there exist $k\in\Z$ and unitary operator $U_{k,m+k}\in\B$ such that the following hold:
	\begin{enumerate}
	\item $||S_{m+n+k-1}\dots S_{m+k}x|| = ||T_{n+k-1}\dots T_kU_{k,m+k}x||$ for all $x\in \H$ and $n \in \N_+$,
	\item $||S_{m-n+k}^*\dots S_{m-1+k}^*x|| = ||T_{-n+k}^*\dots T_{-1+k}^*U_{k,m+k}x||$ for all $x\in \H$ and $n \in \N_+$.
	\end{enumerate}
\end{enumerate}
\end{corollary}

Next result that we will prove is the unitary equivalence of bilateral operator valued weighted shift with the one having positive weights. 
Shields proved in \cite{shields} that each bilateral weighted shift with weights $\{a_n\}_{n\in\Z}\subseteq \C$ is unitarily equivalent to the shift with weights $\{ |a_n|\}_{n\in\Z}$. 
This fact follows from \cite[Theorem~1]{shields}.
Pietrzycki used it to prove that each bounded injective classical bilateral weighted shift $S$ satisfying $S^{*n}S^n=(S^*S)^n$ for any $n\ge 2$ is quasinormal (see \cite[Theorem~3.3]{pietrzycki}). 
Jab\l o\'nski, Jung and Stochel generalized Shields{'} result to the class of weighted shifts on directed trees (see \cite[Theorem~3.2.1]{stochel}).

In the case of bilateral operator valued weighted shifts the situation is more complicated. 
Guyker proved in \cite[Theorem~1]{guyker} that, if $\bshift{S}$ has weights that are commuting and normal operators, then $S$ is unitarily equivalent to the bilateral operator valued weighted shift with weights of the form $(S_n^*S_n)^{\frac{1}{2}}$.
This result is similar to the one of Shields for shifts with classical weights.
Ivanovski mentioned in \cite{ivanovski} that, without loss of generality, we can assume that each bilateral operator valued weighted shift can be assumed to have positive weights. 
He referenced \cite{lambert}. 
However, in \cite{lambert} there is only a proof of unitary equivalence of shifts with those of positive weights for unilateral operator valued weighted shifts with invertible weights.

We will now prove the fact that each bilateral operator valued weighted shift is unitarily equivalent to a bilateral shift with positive weights.
We use argument which is based on similar results from \cite{orovcanec, lambert} for unilateral operator valued weighted shifts.

\begin{theorem}
\label{positive_equivalence}
Let $\bshift{S}$ and $S_n$ be quasi-invertible for all $n \in \Z$. 
Then $S \ue T$, where $\bshift{T}$ and each $T_n$ is positive.
\end{theorem}
\begin{proof}
It follows from the polar decomposition that for each $n\in\Z$ there exist unitary $U_n$ and positive $P_n$ such that $S_n = U_n P_n$.
Let $\tilde{P}$ and $\tilde{U}$ be diagonal operators on $\bspace$ such that $\tilde{P_n} = P_{n+1}$ and $\tilde{U_n} = U_{n+1}$ for all $n\in\Z$. 
Simple calculation can prove that $S=F\tilde U\tilde P$.
It is easy to verify that condition (ii) from Theorem~\ref{main_theorem} is satisfied as $F$ and $F\tilde{U}$ have unitary weights.
Thus there exists a diagonal operator $V\in\B[\bspace]$ such that $VF=F\tilde{U}V$. 
It is true that
\begin{equation*}
S = F\tilde{U}\tilde{P} = VV^*F\tilde{U}\tilde{P} = VFV^*\tilde{P} = V(FV^*\tilde{P}V)V^*.
\end{equation*}
Observe that $V^*\tilde{P}V$ is a diagonal operator. 
This implies that $FV^*\tilde{P}V$ is a bilateral operator valued weighted shift.
Since unitary equivalence preserves positivity and elements of $V^*\tilde{P}V$ are unitarily equivalent to elements of $\tilde{P}$, the proof is completed.
\end{proof}

Now we will state a useful fact which gives necessary conditions of unitary equivalence of bilateral operator valued weighted shifts given by operator of diagonal form.

\begin{lemma}
\label{lemma_norms}
Let $\bshift{S}$, $\bshift{T}$ have quasi-invertible weights. 
Suppose that $S\ue_U T$, where $U$ is of diagonal form and $U_{0,k}\ne 0$ for some $k\in\Z$.
Then $||S_{n+k}|| = ||T_n||$ for each $n\in\Z$.
\end{lemma}
\begin{proof}
Define $V_n=U_{n,n+k}$ for all $n\in\Z$.
By Lemma~\ref{lemat_podstawowy}, $V_{n+1}S_{n+k} = T_nV_{n}$ for each $n\in\Z$, where operators $U_n$ are unitary.
Therefore, we see that
$T_n = V_{n+1} S_{n+k} V_{n}^*$
and
$S_{n+k} = V_{n+1}^* T_n V_{n}$ for each $n\in\Z$. 
This completes the proof.
\end{proof}

In the following proposition we provide necessary condition of unitary equivalence given by a diagonal operator for $\H = \C^2$.

\begin{proposition}
\label{propsition_spectra}
Let $\bshift{S}$, $\bshift{T}$ be defined on $\C^2$ and have normal weights.
Assume that $S\ue_U T$ where $U$ is a diagonal operator.
Then the modulus of eigenvalues of corresponding weights are equal.
\end{proposition}
\begin{proof}
Since all weights are normal matrices, then they are diagonalizable. 
Therefore, it is easy to see that we can diagonalize (using unitary operator) one of the weights in each shifts.
Let $n\in\Z$. By the above we can assume that $S_n$ and $T_n$ are diagonal matrices. 
By Corollary \ref{corollary_main_theorem} there exists unitary $V\in\B[\C^2]$ such that 
$||S_n x|| = ||T_n V x||$ for all $x\in\C^2$.
Let us now assume that 
$$
S_n = 
\begin{bmatrix}
s_1   & 0 \\
  0   & s_2
\end{bmatrix},
\quad
T_n = 
\begin{bmatrix}
t_1   & 0 \\
  0   & t_2
\end{bmatrix},
\quad
V = 
\begin{bmatrix}
v_1   & v_2 \\
v_3   & v_4
\end{bmatrix}
$$

Taking $x = (1,0)$ and $y = (0,1)$, by the previous property, we get the following system of equations: 
\begin{equation*}
\begin{split}
|s_1|^2 = |v_1t_1|^2 + |v_3t_2|^2, \\
|s_2|^2 = |v_2t_1|^2 + |v_4t_2|^2.
\end{split}
\end{equation*}

We see that both equations are convex combinations.
Also, by Lemma~\ref{lemma_norms}, it is true that $\max\{|s_1|, |s_2|\} = \max\{|t_1|, |t_2|\}$. 
Since $V$ is unitary, it must be true that 
\begin{equation*}
\begin{cases}
|s_1| = |t_1|\\
|s_2| = |t_2|
\end{cases}
\text{ or }\quad
\begin{cases}
|s_1| = |t_2|\\
|s_2| = |t_1|
\end{cases}
\end{equation*}
which is exactly our claim.
\end{proof}

We can now use the above result to determine whether two bilateral operator valued weighted shifts on $\C^2$ are unitarily equivalent by a diagonal operator.
First, we use Theorem~\ref{positive_equivalence} to transform both shifts to their forms with positive weights.
Then we compare the eigenvalues of the corresponding weights and check whether their modulus are equal. 
It there is at least one pair of two corresponding weights with at least one different eigenvalue, then it means that eventual unitary equivalence of the shifts cannot be given by a diagonal operator. 
It is important to note that moving to form with positive weights is achieved by using a diagonal operator and, therefore, the argument presented above is correct.

Unfortunately, there is no clear dependency between spectra of weights of original shift and the one with positive weights. 
Another problem is that the condition provided in Proposition~\ref{propsition_spectra} is not sufficient.
To see this let us consider the following

\begin{example}
Let $\H=\C^2$. 
We will set $S_{n} = T_{n} = I$ to be identity operators on $\H$ for $n\in\Z\setminus\{0,1\}$. 
For $n\in\{0,1\}$ we define 
$$
S_n = 
\begin{bmatrix}
s_{1,n}   & 0 \\
  0   & s_{2,n}
\end{bmatrix},
\quad
T_n = 
\begin{bmatrix}
t_{1,n}   & 0 \\
  0   & t_{2,n}
\end{bmatrix}.
$$

Let us fix $|s_{1,0}|=|t_{2,0}|$ and $|s_{2,0}|=|t_{1,0}|$ and $|s_{1,0}| > |s_{2,0}|>0$.
For $S_1$ and $T_1$ we choose $|s_{1,1}|=|t_{1,1}|$ and $|s_{2,1}|=|t_{2,1}|$ and $|s_{1,1}|>|s_{2,1}|>1$.
Now, by Theorem~\ref{main_theorem}, for $S$ and $T$ to be unitarily equivalent by a diagonal operator we need a unitary operator $U\in\B$ such that:
\begin{equation*}
\begin{split}
||S_0x|| = &||T_0Ux||, \\
||S_{1}S_0x|| = &||T_{1}T_0Ux||, \quad x \in \H .
\end{split}
\end{equation*}
But the above cannot be true as first equation determines that $U$ must be equal
$$
U = 
\begin{bmatrix}
 0   & u \\
  v   & 0
\end{bmatrix},
$$
where $|u|=1$ and $|v|=1$. 
In this case, the second equation is not satisfied. 
\end{example}

Li, Ji and Sun proved in \cite[Theorem~2.1]{li_ji_sun} that bilateral weighted shift defined on $\H=\C^k$ is unitarily equivalent to the one with upper triangular weights. 
We will see that, under some additional assumptions, it is possible to prove that some bilateral operator valued weighted shifts are unitarily equivalent to the ones with diagonal weights.

\begin{proposition}
\label{propo}
Let $\bshift{S}$ be a shift in $\bspace[\C^k]$ for $k \ge 2$ with normal and commuting weights. 
Then there is a $\bshift{D}$ such that $S\ue D$ and $D_n$ is a diagonal operator for each $n\in\Z$.
\end{proposition}
\begin{proof}
It is a well-known fact that any set of normal matrices $\{ T_a\}_{a\in A}$ which commutes with each other can be simultaneously diagonalized \text{i.e.} there exists a unitary matrix $V$ such that $VT_a V^*$ is diagonal for each $a\in A$ (see \cite[Theorem~1.3.19]{horn}). 
Now, we see that a diagonal operator consisting of operators $V$ on its diagonal gives unitary equivalence between $S$ and $\bshift{D}$ where each $D_n$ is diagonal for every $n\in\Z$.
\end{proof}

\section{Unitary equivalence - the non-diagonal case}

In this section we focus on investigation of unitary operators that can give unitary equivalence of bilateral operator valued weighted shifts.
Most of the results concern only finite-dimensional Hilbert spaces.

In \cite[Theorem~1]{shields} one can find a proof of the fact that in case of bilateral shifts on $\bspace[\C]$ unitary equivalence is always given by an operator of diagonal form. 
We will now see that there are bilateral operator valued weighted shifts which are unitarily equivalent, but the unitary equivalence is not given by any operator of diagonal form.

\begin{example}
\label{unitary_equiv_ex}
Assume $\H = \C^2$, $w = \hlf - \hlf\text{i}$ and define
\begin{align}
\label{definition_s_n}
s_n = 
	\begin{cases}
		1,			 & \text{if } n = 0, \\
		\frac{1}{n}, & \text{otherwise}.
	\end{cases}
\end{align} 
Let $\bshift{S}$, $\bshift{T}$ have weights
\begin{equation*}
	\begin{gathered}
	S_n := 
	\begin{bmatrix}
	 s_n  & s_n \\
	 -s_n & s_n	
	\end{bmatrix}, \quad
	T_n := 
	\begin{bmatrix}
	s_{n-1}w + s_{n+1}\bar w   & s_{n-1}\bar w + s_{n+1}w \\
	-s_{n-1}\bar w - s_{n+1}w & s_{n-1}w + s_{n+1}\bar w 
	\end{bmatrix}
	\end{gathered}
\end{equation*}
for $n\in\Z$.
It is easy to see that weights of $S$ and $T$ are invertible, bounded and normal.
We construct unitary operator with two nonzero diagonals which gives unitary equivalence of $S$ and $T$.
Let us define the following operators
$$
A = \frac{1}{2}
\begin{bmatrix}
1  & -\text{i} \\
\text{i} & 1
\end{bmatrix},
\quad
B = \frac{1}{2}
\begin{bmatrix}
1  & \text{i} \\
-\text{i} & 1
\end{bmatrix}.
$$
Both $A$ and $B$ are orthogonal projections onto one-dimensional subspaces.
Moreover, $AB=BA=0$ and $A+B=I$. 
Define
$$
U = 
\begin{bmatrix}
 \ddots & \ddots	& \ddots 	& \ddots & \ddots	\\
 \ddots & 0 	 	& A 	 	& 0      & \ddots  	\\
 \ddots & B	   		& \boxed{0}	& A  	 & \ddots	\\
 \ddots & 0 		& B			& 0 	 & \ddots	\\
 \ddots & \ddots 	& \ddots	& \ddots & \ddots	
\end{bmatrix}
$$
The reader can check that $U$ is unitary and $US = TU$.

Now, we show that it is not possible to find unitary operator of diagonal form which would give unitary equivalence of $S$ and $T$.
First, one can easily verify that 
$||S_n|| = \sqrt{2}|s_n|$ for each $n\in\Z$.
Let us now compute the norms of operators $T_n$.
We find the eigenvalues of $T_n^*T_n$ using the characteristic polynomial
$$
W(\lambda) = \lambda^2 - 2\lambda (s_{n-1}^2 + s_{n+1}^2) + 4s_{n-1}^2 s_{n+1}^2.
$$
The roots of $W$ are $2s_{n+1}^2$ and $2s_{n-1}^2$, hence 
$||T_n|| = 
\operatorname{max}\{\sqrt{2}|s_{n-1}|, \sqrt{2}|s_{n+1}|\}$.
Now, it follows from \eqref{definition_s_n} that
\begin{align}
\label{eq_norms}
\left\{\;
\begin{aligned}
 ||S_i|| &= 1 \text{ if and only if } i \in \{ -1, 0, 1\}, \\
 ||T_i|| &= 1 \text{ if and only if } i \in \{ -2, -1, 0, 1, 2\}.
\end{aligned}
\right.
\end{align}

Suppose that, contrary to our claim, $S$ and $T$ are unitarily equivalent by an operator of diagonal form. 
By Lemma~\ref{lemma_norms}, there exists $k\in\Z$ that $||S_{n+k}|| = ||T_n||$ for all $n\in\Z$.
This contradicts \eqref{eq_norms}.
\end{example}

We presented the example of two bilateral operator valued weighted shifts that are unitarily equivalent by an operator that is not of diagonal form.
We also proved that there is no operator of diagonal form that would give this unitary equivalence.
Let us note that, by Proposition~\ref{propo}, this example can be significantly simplified if we diagonalize all weights before performing any computations. 
We leave the details to the reader.

Example \ref{unitary_equiv_ex} shows even more.
Let us first recall some known results.
Shields showed in \cite[Theorem~1]{shields} that, if two bilateral shifts with complex weights are unitarily equivalent, then there exists $k\in\Z$ such that $|s_n| = |t_{n+k}|$ for each $n\in\Z$.
Moreover, it follows from \cite[Theorem~1]{orovcanec} that, if two unilateral shifts $\ushift{S}$, $\ushift{T}$ with quasi-invertible weights are unitarily equivalent, then the unitary equivalence is given by a diagonal operator. 
It follows from similar argument as in Lemma~\ref{lemma_norms} that $||S_n|| = ||T_n||$ for each $n\in\N$. 
This is not true for bilateral operator valued weighted shifts defined on a Hilbert space of dimension greater then one.

We will now investigate unitary operators on $\bspace$ that can give unitary equivalence of bilateral weighted shifts defined on $\H$.
Note that, as in the case of finite-dimensional Hilbert space quasi-invertibility is the same property as invertibility, Corollary \ref{corollary_podstawowy} already gives us the information that, if any element in a matrix representation of a unitary operator is nonzero, then the entire diagonal containing this element is nonzero. 
Therefore, we will focus only on number of nonzero diagonals in unitary operators.

Let $U$ be a unitary operator acting on $\bspace$ with two nonzero diagonals. Then there exist $k_1, k_2 \in\Z$ 
such that $k_1 \ne k_2$ and operators $U_{n,n+k_1}, U_{n,n+k_2}$ are nonzero elements from these diagonals for all $n\in\Z$.
From now we identify nonzero diagonals of such operators with $k_1$ and $k_2$ and denote $U_n^{(1)}:=U_{n,n+k_1}$, $U_n^{(2)}:=U_{n,n+k_2}$ for all $n\in\Z$. 
Without loss of generality we can assume that $k_2 > k_1$.
We generalize this notation to an arbitrary number of diagonals in $U$.

Next proposition states that, if there are exactly two nonzero diagonals in a unitary operator, then both diagonals contain only partial isometries.

\begin{proposition}
\label{proposition_two_diagonals}
Suppose that $U\in\B[\bspace]$ is a unitary operator that has exactly two nonzero diagonals and all other elements are zero operators. 
Then the elements on these diagonals are partial isometries such that elements in each row of $U$ have orthogonal ranges.
\end{proposition}
\begin{proof}
Let us fix $k = k_2 - k_1 > 0$.
For simplicity let us set $A_n:=U_n^{(1)}$, $B_n:=U_n^{(2)}$ for all $n\in\Z$.
Both $\{A_n\}_{n\in\Z}$ and $\{B_n\}_{n\in\Z}$ are sequences of bounded operators.
Note that $U$ is unitary if and only if conditions
\begin{subequations}
	\begin{align}
		I &= A_n A_n^* + B_n B_n^*, \label{2_diag_1} \\
		I &= A_{n+k}^* A_{n+k} + B_n^* B_n, \\
		0 &= A_{n+k}B_n^*, \label{2_diag_3} \\
		0 &= A_n^*B_n \label{2_diag_4}, 
	\end{align}
\end{subequations}
hold for all $n\in\Z$.
Now, we can multiply \eqref{2_diag_1} by $A_n$ from the right and get
\begin{equation}
\label{eq_2}
A_n = A_n A_n^* A_n + B_n B_n^* A_n.
\end{equation}
Now, by \eqref{2_diag_4} and \eqref{eq_2}, we see that $B_n B_n^* A_n = 0$ and thus, by Lemma~\ref{lemma_isometries}, $A_n$ is a partial isometry for all $n\in\Z$. 
It is clear that operators $B_n$ are also partial isometries.
From \eqref{2_diag_4} we deduce that $\Rng(A_n)$ is orthogonal to $\Rng(B_n)$ for all $n\in\Z$.
This completes the proof.
\end{proof}

It is worth noting that, using the property \eqref{2_diag_3}, we can deduce that $\Rng(B_n^*) \perp \Rng(A_{n+k}^*)$ for all $n\in\Z$.

Next example shows that sequences $\{A_n\}_{n\in\Z}$, $\{B_n\}_{n\in\Z}$ do not need to be sequences of orthogonal projections.

\begin{example}
Let $\H=\C^2$. 
We define the following
$$
A_n = 
\begin{bmatrix}
  0 & a_n \\
  0 & 0
\end{bmatrix},
\quad
B_n = 
\begin{bmatrix}
  0   & 0 \\
 b_n  & 0
\end{bmatrix},
$$
where $|a_n| = |b_n| = 1$ for all $n\in\Z$.
The reader can verify that these operators satisfy conditions \eqref{2_diag_1} - \eqref{2_diag_4} from the proof of Proposition~\ref{proposition_two_diagonals} and form a unitary operator $U$ with $k_1=-1$ and $k_2=1$.
Now we define $\bshift{S}$ in the following way
$$
S_n = 
\begin{bmatrix}
s_{1,n}   & 0 \\
  0   & s_{2,n}
\end{bmatrix},
$$
where $s_{1,n} s_{2,n} \ne 0$ for all $n\in\Z$. 
It is easy to check that $USU^*$ is a bilateral operator valued weighted shift and neither $A_n$ nor $B_n$ are orthogonal projections for any $n\in\Z$.
\end{example}

Now we will prove useful lemma that we will use later in the paper.
We provide more general version than we need, which is true for an arbitrary nonzero Hilbert space.

\begin{lemma}
\label{lemma_orthogonal}
Let $\H$ be a Hilbert space and $n\in\N_+$. 
Assume that $A_i\in\B$ are positive operators for $i\in\{1,\dots,n\}$ such that 
\begin{equation}
\label{equation_sum_of_orthogonal}
C := A_1 + \dots + A_n
\end{equation} 
and $\dim\Rng(C)=1$.
Then $A_i = a_iC$, where $a_i\in\R_+$ for all $i\in\{1,\dots,n\}$ and $\sum_{i=1}^n a_i=1$.
\end{lemma}
\begin{proof}
Let $M:=\Rng(C)=\operatorname{lin}\{\bar e\}$ for some normalized $\bar e\in\H$. 
Set $B$ to be an orthonormal basis of $\H$ containing $\bar e$.
Then, by \eqref{equation_sum_of_orthogonal}, we know that 
\begin{equation*}
\langle A_ie, e\rangle = 0, \quad e\in B\setminus\{\bar e\}, i\in\{1\dots,n\}.
\end{equation*}
This, Cauchy-Schwarz inequality and the square root theorem imply that 
\begin{align*}
|\langle A_i e, e' \rangle|^2 & = |\langle A_i^\hlf e, A_i^\hlf e' \rangle|^2 \\ 
& \le \langle A_i e, e\rangle \langle A_i e', e' \rangle = 0, 
\quad e,e'\in\B, (e,e')\ne(\bar e, \bar e). 
\end{align*}
Thus $A_ie = 0$ for all $i\in\{1\dots,n\}$ and $e\in B\setminus\{\bar e\}$.
Hence for all $i\in\{1\dots,n\}$, $A_i = a_iC$ for some $a_i \in\R_+$.
Now, it follows from \eqref{equation_sum_of_orthogonal} that $\sum_{i=1}^n a_i=1$.
\end{proof}

We will state another lemma which gives an equivalent condition for an operator $U\in\B[\bspace]$ with three nonzero diagonals to be a unitary operator.
It is worth noting that this result can be generalized to arbitrary diagonals, but then it is significantly more complicated.
Thus we present it only for the case in which the three diagonals are located next to each other in the center of the matrix of $U$ (see \eqref{three_diagonal_operator} below).
We leave its proof to the reader.

\begin{lemma}
Assume that $U\in\B[\bspace]$ is an operator of the form
\begin{equation}
\label{three_diagonal_operator}
U = 
\begin{bmatrix}
 \ddots & \ddots	& \ddots 	& \ddots & \ddots	\\
 \ddots & B_{-1} 	& C_{-1} 	& 0      & \ddots  	\\
 \ddots & A_0	   	& \boxed{B_0}	& C_0  	 & \ddots	\\
 \ddots & 0 		& A_1		& B_1 	 & \ddots	\\
 \ddots & \ddots 	& \ddots	& \ddots & \ddots	
\end{bmatrix},
\end{equation}
where $\{A_n\}_{n\in\Z}$, $\{B_n\}_{n\in\Z}$, $\{C_n\}_{n\in\Z}\subseteq\B$.
Then $U$ is unitary if and only if the following
\begin{subequations}
	\begin{align}
		I &= A_n A_n^* + B_n B_n^* + C_n C_n^*, \label{3_diag_1} \\
		0 &= C_n A_{n+2}^*, \label{3_diag_2} \\
		0 &= A_{n+1} B_n^* + B_{n+1} C_n^*, \label{3_diag_3} \\
		I &= A_{n+2}^* A_{n+2} + B_{n+1}^* B_{n+1} + C_n^* C_n, \label{3_diag_4} \\
		0 &= A_{n}^* C_n, \label{3_diag_5} \\
		0 &= C_n^* B_n + B_{n+1}^* A_{n+1}, \label{3_diag_6}
	\end{align}
\end{subequations}
hold for all $n\in\Z$.
\end{lemma}

The next proposition states that, under some additional assumptions, a unitary operator $U$ defined by \eqref{three_diagonal_operator} may consist of at most two nonzero diagonals.

\begin{proposition}
\label{proposition_isometry}
Assume that $\H$ is two-dimensional, $U\in\B[\bspace]$ is a unitary operator of the form \eqref{three_diagonal_operator}
and $1\in\spec(C_kC_k^*)\cap\spec(C_{k+1}C_{k+1}^*)$ for some $k\in\Z$.
Let $S$ and $USU^*$ be bilateral operator valued weighted shifts with invertible weights.
Then at least one of the sequences $\{A_n\}_{n\in\Z}$, $\{B_n\}_{n\in\Z}$ or $\{C_n\}_{n\in\Z}$ consists of zero operators only.
\end{proposition}
\begin{proof}
Note that, by Corollary \ref{corollary_podstawowy}, if any element of any of the three sequences is the zero operator, then all the operators in this sequence are zero operators.

First, we assume that $C_n \ne 0$ for all $n\in\Z$. Otherwise, the proof is completed.
Now, let $\dim\Rng(C_n) = 2$ for some $n\in\Z$. 
Since $C_n$ is invertible, then, by \eqref{3_diag_5}, we get that $A_n^* = 0$ which means that $A_n = 0$.

Now assume that $\dim\Rng(C_n) = 1$ for all $n\in\Z$. 
Let $n\in\{k, k+1\}$.
Since $1\in\spec(C_nC_n^*)$, then $Z_n:=I-C_nC_n^*$ is not invertible positive operator.
Then, by \eqref{3_diag_1}, we see that
$$
	A_n A_n^* + B_n B_n^* = Z_n
$$
so, by Lemma~\ref{lemma_orthogonal}, we have $A_n A_n^* = a_n Z_n$ and $B_n B_n^* = b_n Z_n$, where $a_n+b_n = 1$.
If $a_nb_n=0$ for any $n\in\{k, k+1\}$, then the proof is completed.
Otherwise, by \eqref{3_diag_5}, for $n\in\{k, k+1\}$,
\begin{equation}
\label{1235113}
R(C_n) \perp R(A_n) = R(A_nA_n^*) = R(B_nB_n^*) = R(B_n).
\end{equation}
Hence \eqref{3_diag_6} with $n=k$ implies $0 = B_{k+1}^* A_{k+1}$ which, together with $\eqref{1235113}$ for $n=k+1$ and the fact that $\Rng(A_{k+1})\ne \{0\}$, lead to contradiction.
\end{proof}

Observe that Proposition~\ref{proposition_isometry} remains true, if we replace the assumption that $1\in\spec(C_kC_k^*)\cap\spec(C_{k+1}C_{k+1}^*)$ by the assumption that $1\in\spec(A_k^*A_k)\cap\spec(A_{k+1}^*A_{k+1})$ for some $k\in\Z$.

\begin{proposition}
Let $\H$ be two-dimensional Hilbert space and $U$ be a unitary operator defined as in \eqref{three_diagonal_operator}.
Assume that $S$ and $USU^*$ are bilateral operator valued weighted shifts with invertible weights and $\dim\Rng(B_n) \le 1$ for any $n\in\Z$.
Then for all $n\in\Z$ either $A_n$ or $C_n$ is a partial isometry.
\end{proposition}
\begin{proof}
Note that the same argument as in Proposition~\ref{proposition_isometry} can be used to cover the case when $\dim\Rng(C_n)\ne 1$ for any $n\in\Z$.

Let us now assume that $\dim\Rng(C_n) = 1$ for all $n\in\Z$.
Fix $n\in\Z$. 
By the above it us true that $\dim\Rng(B_nB_n^*C_n)$ is equal to $0$ or $1$.
If it is equal to $1$, then we see that $\Rng(B_nB_n^*C_n) = \Rng(B_nB_n^*)$,
as $\dim \Rng(B_nB_n^*)\le 1$.
Therefore $A_n^*B_nB_n^* = 0$.
Hence $A_n$ is a partial isometry.
Now assume $\dim\Rng(B_nB_n^*C_n) = 0$. 
This implies that $C_n$ is a partial isometry.
\end{proof}

The next result states that there cannot be more then $m$ nonzero diagonals which contain partial isometries in unitary operator giving unitary equivalence of bilateral operator valued weighted shifts defined on $m$-dimensional Hilbert space for $m\ge 2$.

\begin{proposition}
\label{proposition_general}
Let $\H$ be a $m$-dimensional Hilbert space for $m\ge2$ and let $\bshift{S}$, $\bshift{T}$ have quasi-invertible weights.
Assume $U\in\B[\bspace]$ is unitary and its matrix representation consists of partial isometries only.
If $US=TU$, then $U$ has at most $m$ nonzero diagonals and all other elements of $U$ are zero operators.
\end{proposition}
\begin{proof}
By the fact that $UU^* = I$ we get that $\sum_{j\in\Z} U_n^{(j)} (U_n^{(j)})^* = I$ for all $n\in\Z$.
It follows from the Lemma~\ref{lemma_isometries} that $P_n^{(j)}:=U_n^{(j)} (U_n^{(j)})^*$ is an orthogonal projection for all $j\in\Z$.
It is a well-known fact that, if $\sum_{j\in\Z} P_n^{(j)}$ is an orthogonal projection, then $\Rng(P_n^{(i)})\perp \Rng(P_n^{(j)})$ for all $i$,$j\in\Z$ such that $i\ne j$.
The rest follows directly from the fact that $\H$ is $m$-dimensional and from Corollary \ref{corollary_podstawowy}.
\end{proof}

The next example shows that it is possible to find unitary operator with three nonzero diagonals that give unitary equivalence between bilateral operator valued weighted shifts defined on $\C^2$.

\begin{example}
\label{example_three_trivial}
\def\shlf{\frac{1}{\sqrt{2}}}
Assume $\H = \C^2$.
First, let us define unitary operator $U$ of the form \eqref{three_diagonal_operator} with three nonzero diagonals, where
$$
A_{2n} = 
\begin{bmatrix}
  \shlf & 0 \\
  0 & 0
\end{bmatrix},
\quad
A_{2n+1} = 
\begin{bmatrix}
  0 & 0 \\
  0 & -\shlf
\end{bmatrix},
\quad 
B_n = 
\begin{bmatrix}
  \shlf & 0 \\
  0 & \shlf
\end{bmatrix},
\quad
C_{n+1} = -A_n,
$$
for all $n\in\Z$.
It can be verified that conditions \eqref{3_diag_1} - \eqref{3_diag_6} are satisfied.
Hence $U$ is a unitary operator.
Let us now define $\bshift{S}$ in the following way
$$
S_{n} = 
\begin{bmatrix}
  0 & (-1)^{n} \\
  (-1)^{n} & 0
\end{bmatrix},
\quad n\in\Z.
$$
The reader can check that $US = SU$.
\end{example}

\section{Further remarks}

Example \ref{example_three_trivial} shows that it is possible to find unitary operator with three nonzero diagonals, which gives unitary equivalence of bilateral weighted shifts defined on $\C^2$, however, this unitary equivalence can be given by the identity operator. 
It is an open question, whether for $\bshift{S}$, $\bshift{T}$, where $\H$ is a two-dimensional Hilbert space, $S\ue T$ implies that there exists $U$ that has at most two nonzero diagonals with all other elements of $U$ being zero operators such that $S\ue_U T$.
If one proves that any unitary equivalence of bilateral shifts defined on finite-dimensional Hilbert space can be given by an operator consisting only of partial isometries, then Proposition \ref{proposition_general} gives the positive answer.

Another interesting problem for further investigation, which comes up naturally, is the problem of characterization of unitary equivalence of bilateral shifts defined on finite-dimensional Hilbert space.
Corollary \ref{corollary_main_theorem} gives characterization of unitary equivalence given only by an operator of diagonal form.
Example \ref{unitary_equiv_ex} shows that there is a rich class of unitary operators in $\bspace[\C^k]$ which are not of diagonal form and can give unitary equivalence of bilateral weighted shifts.
Clearly, we see that the problem of complete characterization is more complicated than in case of unilateral operator valued weighted shifts and bilateral weighted shifts having classical weights.

\section*{Acknowledgments}

The author would like to thank Zenon Jab\l o\'nski for insightful discussions concerning the subject of the paper.

\end{document}